\documentclass[a4paper,11pt,leqno]{article}

\usepackage[latin1]{inputenc}
\usepackage[T1]{fontenc} 
\usepackage[english]{babel}
\usepackage{verbatim}
\usepackage{amsmath}
\usepackage{amsthm}
\usepackage{amssymb}
\usepackage{bm}

\theoremstyle{definition}
\newtheorem{defin}{Definition}[section]

\newtheorem{rem}[defin]{Remark}

\theoremstyle{plane}
\newtheorem{theo}[defin]{Theorem}
\newtheorem{prop}[defin]{Proposition}
\newtheorem{coroll}[defin]{Corollary}
\newtheorem{lemma}[defin]{Lemma}

\newcommand{\mbb}{\mathbb}

\newcommand{\mc}{\mathcal}
\newcommand{\veps}{\varepsilon}
\newcommand{\what}{\widehat}
\newcommand{\wtilde}{\widetilde}
\newcommand{\vphi}{\varphi}
\newcommand{\oline}{\overline}

\newcommand{\hra}{\hookrightarrow}
\newcommand{\lra}{\Leftrightarrow}
\newcommand{\g}{\gamma}

\newcommand{\R}{\mathbb{R}}

\newcommand{\N}{\mathbb{N}}
\newcommand{\Z}{\mathbb{Z}}

\renewcommand{\Re}{{\rm Re}\,}

\def\d{\partial}

\title{\large{\bfseries{TIME-DEPENDENT LOSS OF DERIVATIVES FOR \\ HYPERBOLIC OPERATORS WITH
NON REGULAR COEFFICIENTS}}}

\author{\textsl{Ferruccio Colombini} \\
\small{\textsc{Universit\`a di Pisa}} -- \small{\ttfamily{colombini@dm.unipi.it}} \vspace{0.3cm} \\
\textsl{Daniele Del Santo} \\
\small{\textsc{Universit\`a di Trieste}} -- \small{\ttfamily{delsanto@units.it}} \vspace{0.3cm} \\
\textsl{Francesco Fanelli} \\
\small{\textsc{BCAM - Basque Center for Applied Mathematics}} --
\small{\ttfamily{ffanelli@bcamath.org}} \vspace{0.3cm} \\
\textsl{Guy M\'etivier} \\
\small{\textsc{Universit\'e de Bordeaux 1}} -- \small{\ttfamily{guy.metivier@math.u-bordeaux1.fr}}}

\date\today

\begin{document}

\maketitle

\subsubsection*{Abstract}
{\small In this paper we will study the Cauchy problem for strictly hyperbolic operators with low regularity
coefficients in any space dimension $N\geq1$. We will suppose the coefficients to be log-Zygmund continuous in time and log-Lipschitz continuous in space.
Paradifferential calculus with parameters will be the main tool to get energy estimates in Sobolev spaces and these estimates will present a time-dependent loss of derivatives.}

\subsubsection*{Mathematical Subject classification (2010)}{\small Primary: 35L15; Secondary: 35B65, 35S50, 35B45}

\subsubsection*{Keywords}{\small Hyperbolic operators, non-Lipschitz coefficient, log-Zygmund regularity, energy estimates, well-posedness}

\section{Introduction}

This paper is devoted to the study of the Cauchy problem for a second order strictly hyperbolic operator defined in a strip
$[0,T]\times\R^N$, for some $T>0$ and  $N\geq1$. Consider a second order operator of the form
\begin{equation} \label{def:op}
Lu\;:=\;\d^2_tu\,-\,\sum_{j,k=1}^N\d_j\left(a_{jk}(t,x)\,\d_ku\right)
\end{equation}
(with $a_{jk}=a_{kj}$ for all $j$, $k$) and assume that $L$ is strictly hyperbolic with bounded coefficients, i.e. there exist
two constants $0<\lambda_0\leq\Lambda_0$ such that
$$
\lambda_0\,|\xi|^2\;\leq\;\sum_{j,k=1}^N a_{jk}(t,x)\,\xi_j\,\xi_k\;\leq\;\Lambda_0\,|\xi|^2
$$
for all $(t,x)\in[0,T]\times\R^N$ and all $\xi\in\mbb{R}^N$.

It is well-known (see \cite{H-S}; see also e.g. \cite[Ch. IX]{Horm} or \cite[Ch. 6]{Miz} for analogous results) that, if the coefficients $a_{jk}$ are Lipschitz continuous
with respect to $t$ and only measurable in $x$, then the Cauchy problem for $L$ is well-posed in
$H^1\times L^2$. If the $a_{jk}$'s are Lipschitz continuous
with respect to $t$ and $\mc{C}^\infty _b$ (i.e. $\mc{C}^\infty$ and bounded with all their derivatives)
with respect to the space variables, one can recover the well-posedness in $H^{s+1}\times H^s$ for all $s\in\R$.
Moreover, in the latter case, one gets, for all $s\in\R$  and for a constant $C_s$ depending only on it,  the following
energy estimate:
\begin{eqnarray}
 &&\sup_{0\le t \le T} \biggl(\|u(t, \cdot)\|_{H^{s+1}}\,  +
 \|\partial_t u(t,\cdot)\|_{H^s}\biggr) \label{est:no-loss} \\
&&\qquad\qquad\qquad\qquad\qquad\qquad
\leq\,C_s\,\,\left(\|u(0, \cdot)\|_{H^{s+1}}+
 \|\partial_t u(0,\cdot)\|_{H^s} + \int_0^{T}  \|  L u(t,\cdot)\|_{H^s}\, dt\right) \nonumber
\end{eqnarray}
for all $u\in\mc{C}([0,T];H^{s+1}(\R^N))\,\cap\,\mc{C}^1([0,T];H^s(\R^N))$ such that $Lu\in L^1([0,T];H^s(\R^N))$.
Let us explicitly remark that previous inequality involves no loss of regularity for the function $u$:
estimate \eqref{est:no-loss} holds for every $u\in\mc{C}^2([0,T];H^\infty(\R^N))$ and the Cauchy problem for $L$ is
well-posed in $H^\infty$ \emph{with no loss of derivatives}.

If the Lipschitz continuity (in time) hypothesis is not fulfilled, then \eqref{est:no-loss} is no more true. Nevertheless,
one can still try to recover $H^\infty$-well-posedness, possibly with a {\it loss of derivatives} in the energy estimate. 

The first case to consider is the case of the coefficients $a_{jk}$ depending only on $t$:
$$
 Lu\;=\;\d^2_tu\,-\,\sum_{j,k=1}^N a_{jk}(t)\,\d_j\d_ku\,.
$$
In \cite{C-DG-S}, Colombini, De Giorgi and Spagnolo assumed the coefficients to satisfy an integral log-Lipschitz condition:
\begin{equation} \label{hyp:int-LL}
\int^{T-\veps}_0\left|a_{jk}(t+\veps)\,-\,a_{jk}(t)\right|dt\;\leq\;C\,\veps\,\log\left(1\,+\,\frac{1}{\veps}\right)\,,
\end{equation}
for some constant $C>0$ and all $\veps\in\,]0,T]$. To get the energy estimate, they first smoothed out the
coefficients using a mollifier kernel $\left(\rho_\veps\right)_\veps$. Then, by Fourier transform, they defined
an approximated energy $E_\veps(\xi,t)$ in phase space, where the problem becomes a family of  ordinary differential equations. At that point,
the key idea was to perform a different approximation of the coefficients in different zones of the phase space: in particular,
they set $\veps\,=\,|\xi|^{-1}$. Finally, they obtained an energy estimate \emph{with a fixed loss of derivatives}: there exists
a constant $\delta>0$ such that, for all $s\in\R$, the inequality
\begin{eqnarray}
  &&\sup_{0\le t \le T} \biggl(\|u(t, \cdot)\|_{H^{s+1-\delta}}\,  +
 \|\partial_t u(t,\cdot)\|_{H^{s-\delta}}\biggr) \label{est:c-loss} \\
&&\qquad\qquad\qquad\qquad\qquad\qquad\leq\,C_s\,\left(\|u(0, \cdot)\|_{H^{s+1}}+
 \|\partial_t u(0,\cdot)\|_{H^s} + \int_0^{T}  \|  L u(t,\cdot)\|_{H^s}\, dt\right) \nonumber
\end{eqnarray}
holds true for all $u\in\mc{C}^2([0,T];H^\infty(\R^N))$, for some constant $C_s$ depending only on $s$. Let us remark that if the
coefficients $a_{jk}$ are not Lipschitz continuous then a loss of regularity cannot be avoided, as shown
by Cicognani and Colombini in \cite{Cic-C}.
Besides, in that paper the authors proved that, if the regularity of the coefficients $a_{jk}$ is measured by a
modulus of continuity, any intermediate modulus of continuity between the Lipschitz and the
log-Lipschitz ones entails necessarily a loss of regularity, which, however, can be
made arbitrarily small.

Recently Tarama (see paper \cite{Tar}) analysed the problem when coefficients satisfy an integral
log-Zygmund condition: there exists a constant $C>0$ such that, for all $j$, $k$ and all $\veps\in\,]0,T/2[$, one has
\begin{equation} \label{hyp:int-LZ}
 \int^{T-\veps}_\veps\left|a_{jk}(t+\veps)\,+\,a_{jk}(t-\veps)\,-\,2\, a_{jk}(t)\right|dt\;\leq\;
C\,\veps\,\log\left(1\,+\,\frac{1}{\veps}\right)\,.
\end{equation}
On the one hand, this condition is somehow related, for $a$ of class $\mc{C}^2([0,T])$, to the pointwise condition
$|a(t)|+|t\,a'(t)|+|t^2\,a''(t)|\,\leq\,C$, considered by Yamazaki in \cite{Yama}. On the other hand,
it's obvious that, if the $a_{jk}$'s satisfy \eqref{hyp:int-LL}, then they satisfy  also \eqref{hyp:int-LZ}: so,
a more general class of functions is considered. Again, Fourier transform, smoothing out the cofficients and
linking the approximation parameter with the dual variable were fundamental tools in the analysis of Tarama.
The improvement with respect to paper \cite{C-DG-S}, however, was obtained defining a new energy, which involved
(by differentiation in time) second derivatives of the approximated coefficients. Finally,
he got an estimate analogous to \eqref{est:c-loss}, which implies, in particular, well-posedness
in the space $H^\infty$.

In paper \cite{C-L}, Colombini and Lerner considered instead the case in which coefficients $a_{jk}$ 
depend both in time and in space variables.  In particular, they assumed an isotropic punctual log-Lipschitz
condition, i.e. there exists a constant $C>0$ such that, for all $\zeta=(\tau,\xi)\in\R\times\R^N$, $\zeta\neq0$, one has
$$
\sup_{z=(t,x)\in\R\times\R^N}\;\left|a_{jk}(z+\zeta)\,-\,a_{jk}(z)\right|\;\leq\;C\,|\zeta|\,
\log\left(1\,+\,\frac{1}{|\zeta|}\right)\,.
$$
Again, smoothing coefficients with respect to the time variable is required; on the contrary, one cannot use the Fourier transform, due to the dependence of $a_{jk}$ on $x$. The authors bypassed this problem taking advantage of Littlewood-Paley decomposition and paradifferential calculus.
Hence, they considered the energy concerning each localized part $\Delta_\nu u$ of
the solution $u$, and then they performed a weighted summation to put all these pieces together.
Also in this case, they had to consider
a different approximation of the coefficients in different zones of the phase space, which was obtained setting
$\veps=2^{-\nu}$ (recall that $2^\nu$ is the ``size'' of the frequencies in the $\nu$-th ring, see subsection \ref{ss:L-P}
below). In the end, they got the following statement: for all $s\in\,]0,1/4]$, there exist positive constants $\beta$ and $C_s$ and
a time $T^*\in\,]0,T]$ such that
\begin{eqnarray}
  &&\sup_{0\le t \le T^*} \biggl(\|u(t, \cdot)\|_{H^{-s+1-\beta t}}\,  +
 \|\partial_t u(t,\cdot)\|_{H^{-s-\beta t}}\biggr) \label{est:t-loss} \\
&&\qquad\qquad\qquad\qquad\qquad
\leq\,C_s\,\left(\|u(0, \cdot)\|_{H^{-s+1}}+
 \|\partial_t u(0,\cdot)\|_{H^{-s}} + \int_0^{T^*}  \|  L u(t,\cdot)\|_{H^{-s-\beta t}}\, dt\right) \nonumber
\end{eqnarray}
for all $u\in\mc{C}^2([0,T];H^\infty(\R^N))$.
Let us point out that the bound on $s$ was due to this reason: the
product by a log-Lipschitz function is well-defined in $H^s$ if and only if $|s|<1$. Note also that this fact gives us a bound
on the lifespan of the solution: the regularity index $-s+1-\beta T^*$ has to be strictly positive, so one can
expect only local in time existence of a solution.
Moreover in the case the coefficients $a_{jk}$ are $\mc{C}^\infty_b$ in space,
the authors proved inequality \eqref{est:t-loss} for all $s$: so,
they still got well-posedness in $H^\infty$, but \emph{with a loss of derivatives increasing in time}.

The case of a complete strictly hyperbolic second order operator,
$$
Lu\,=\,\sum_{j,k=0}^N\d_{y_j}\left(a_{jk}\,\d_{y_k}u\right)\,+\,
\sum_{j=0}^N\left(b_j\,\d_{y_j}u\,+\,\d_{y_j}\!\left(c_j\,u\right)\right)\,+\,d\,u
$$
(here we set $y=(t,x)\in\R_t\times\R^N_x$),
was considered by Colombini and M\'etivier in \cite{C-M}. They assumed the same isotropic log-Lipschitz condition of \cite{C-L}
on the coefficients of the second order part of $L$, while $b_j$ and $c_j$ were supposed to be $\alpha$-H\"older continuous
(for some $\alpha\in\,]1/2,1[\,$) and $d$ to be only bounded. The authors headed towards questions such as
local existence and uniqueness, and also finite propagation speed for local solutions.

Recently, Colombini and Del Santo, in \cite{C-DS} (for a first approach to the problem see also \cite{DS}, where smoothness
in space was required),
came back to the Cauchy problem for the operator \eqref{def:op}, mixing up
a Tarama-like hypothesis (concerning the dependence on the time variable) with the  one of Colombini and Lerner (with respect to $x$).
More precisely, they assumed a pointwise log-Zygmund condition in time and a pointwise log-Lipschitz condition in space,
uniformly with respect to the other variable (see relations \eqref{h:LZ_t} and \eqref{h:LL_x} below).
However, they had to restrict themselves to the case of
space dimension $N=1$: as a matter of fact, a Tarama-kind energy was somehow necessary to compensate the bad behaviour of
the coefficients with respect to $t$, but it was not clear how to define it in higher space dimensions.
Again, localizing energy by Littlewood-Paley decomposition and linking approximation parameter and dual variable lead to
an estimate analogous to \eqref{est:t-loss}.

The aim of the present paper is to extend the result of Colombini and Del Santo to any dimension $N\geq1$. As just pointed out,
the main difficulty was to define a suitable energy related to the solution. So, the first step
is to pass from functions $a(t,x)$ with low regularity modulus of continuity, to more general symbols $\sigma_a(t,x,\xi)$
(obviously related to the initial function $a$) satisfying the same hypothesis in $t$ and $x$, and then to
consider paradifferential operators associated to these symbols. Nevertheless, positivity hypothesis on $a$
(required for defining a strictly hyperbolic problem) does not translate, in general, to positivity of the corresponding operator,
which is fundamental in obtaining energy estimates. At this point, paradifferential calculus depending
on a parameter $\g\geq1$, defined and developed by M\'etivier in \cite{M-1986} (see also \cite{M-Z}),
comes into play and allows us to
recover positivity of the (new) paradifferential operator associated to $a$. Defining a localized energy and an
approximation of the coefficients depending on the dual variable are, once again, basic ingredients in closing estimates.
Hence, in the end we will get an inequality similar to \eqref{est:t-loss}, for any $s\in\,]0,1[$.

\medskip
The paper is organized as follows.

First of all, we will introduce the work hypothesis for our strictly hyperbolic problem, and we will state our main results.

Then, we will present the tools we need, all from Fourier Analysis. In particular, we will
recall Littlewood-Paley decomposition and some results about (classical) paradifferential calculus, as introduced first
by J.-M. Bony in the famous paper \cite{Bony}. We will need also to define a different class of Sobolev spaces, of
logarithmic type, as done in \cite{C-M}: they will come into play in our computations. Moreover, we will
present also paradifferential calculus depending on a parameter (which is basic in our analysis, as already pointed out),
as introduced in \cite{M-1986} and \cite{M-Z}.
A complete treatement about functions with low regularity modulus of continuity will end this section. In particular, we
will focus on log-Zygmund and log-Lipschitz conditions: taking
advantage of paradifferential calculus, we will establish some properties of functions satisfying such hypothesis.
Hence, we will pass to consider more general symbols and the
associated paradifferential operators, for which we will develop also a symbolic calculus and we will state a
fundamental positivity estimate.

This having been done, we will be then ready to tackle the proof of our main result, for which we will go back to the main ideas of
paper \cite{C-DS}. First of all, by use of a convolution kernel, we will smooth out the coefficients, but with respect to
the time variable only. As a matter of facts, low regularity in $x$ will be compensated by considering paradifferential
operators associated to our coefficients. Then, we will decompose the solution $u$ to the Cauchy problem for \eqref{def:op}
into dyadic blocks $\Delta_\nu u$, for which we will define an approximate localized energy $e_\nu$: the dependence on
the approximation parameter $\veps$ will be linked to the phase space localization, setting $\veps=2^{-\nu}$. The piece of energy
$e_\nu$ will be of Tarama type, but this time multiplication by functions will be replaced by action of paradifferential
operators associated to them. A weighted summation of these pieces will define the total energy $E(t)$ associated to $u$.
The rest of the proof is classical: we will differentiate $E$ with respect to time and, using Gronwall Lemma, we will get
a control for it in terms of the initial energy $E(0)$ and the external force $Lu$ only.

\subsection*{Acknowledgments}
The work was mostly prepared during the Ph.D. thesis of the third author at SISSA and Universit\'e Paris-Est: he is grateful to both these institutions. During the final part, the third author was partially supported by Grant MTM2011-29306-C02-00, MICINN, Spain, ERC Advanced Grant FP7-246775 NUMERIWAVES, ESF Research Networking Programme OPTPDE and Grant PI2010-04 of the Basque Government.

\section{Basic definitions and main result} \label{s:results}

This section is devoted to the presentation of our work setting and of our main results.

Let us consider the operator over $[0,T_0]\times\R^N$ (for some $T_0>0$ and $N\geq1$)
\begin{equation} \label{eq:op}
 Lu\,=\,\d^2_tu\,-\,\sum_{i,j=1}^N\d_i\left(a_{ij}(t,x)\,\d_ju\right)\,,
\end{equation}
and let us suppose $L$ to be strictly hyperbolic with bounded coefficients, i.e. there exist two positive constants
$0<\lambda_0\leq\Lambda_0$ such that, for all $(t,x)\in\mbb{R}_t\times\mbb{R}^N_x$ and all $\xi\in\mbb{R}^N$, one has
\begin{equation} \label{h:hyp}
 \lambda_0\,|\xi|^2\,\leq\,\sum_{i,j=1}^N a_{ij}(t,x)\,\xi_i\,\xi_j\,\leq\,\Lambda_0\,|\xi|^2\,.
\end{equation}

Moreover, let us suppose the coefficients to be log-Zygmund continuous in the time variable $t$, uniformly with respect to $x$, and
log-Lipschitz continuous in the space variables, uniformly with respect to $t$. This hypothesis reads as follow: there exists a
constant $K_0$ such that, for all $\tau>0$ and all $y\in\mbb{R}^N\setminus\{0\}$, one has
\begin{eqnarray}
 \sup_{(t,x)}\left|a_{ij}(t+\tau,x)+a_{ij}(t-\tau,x)-2a_{ij}(t,x)\right| & \leq &
K_0\,\tau\,\log\left(1\,+\,\frac{1}{\tau}\right) \label{h:LZ_t} \\
 \sup_{(t,x)}\left|a_{ij}(t,x+y)-a_{ij}(t,x)\right| & \leq & K_0\,|y|\,\log\left(1\,+\,\frac{1}{|y|}\right). \label{h:LL_x}
\end{eqnarray}

Now, let us state our main result, i.e. an energy estimate for the operator \eqref{eq:op}.

\begin{theo} \label{t:en-est}
 Let us consider the operator $L$ defined in \eqref{eq:op}, and let us suppose $L$ to be strictly hyperbolic, i.e.
relation \eqref{h:hyp} holds true. Moreover, let us suppose that coefficients $a_{ij}$ satisfy both conditions
\eqref{h:LZ_t} and \eqref{h:LL_x}.

Then, for all fixed $\theta\in\,]0,1[\,$, there exist a $\beta^*>0$, a time $T>0$ and a constant $C>0$ such that the following
estimate,
\begin{eqnarray}
 &&\sup_{0\le t \le T} \biggl(\|u(t, \cdot)\|_{H^{-\theta+1-\beta^* t}}\,  +
 \|\partial_t u(t,\cdot)\|_{H^{-\theta-\beta^* t}}\biggr) \label{est:thesis} \\
&&\qquad\qquad\qquad\qquad\qquad
\leq\,C\,\left(\|u(0, \cdot)\|_{H^{-\theta+1}}+
 \|\partial_t u(0,\cdot)\|_{H^{-\theta}} + \int_0^{T}  \|  L u(t,\cdot)\|_{H^{-\theta-\beta^* t}}\, dt\right)\,, \nonumber
\end{eqnarray}
holds true for all $u\in\mc{C}^2([0,T];H^\infty(\R^N))$.
\end{theo}


So, it's possible to control the Sobolev norms of solutions to \eqref{eq:op} in terms of those of initial data and of
the external force only: the price to pay is a loss of derivatives, increasing (linearly) in time. 




\section{Tools}

In this section we will introduce the main tools, basically from Fourier Analysis, we will need to prove our statement.

First of all, we will recall classical Littlewood-Paley decomposition and some basic results
on dyadic analysis. By use of it, we will also define a different class of Sobolev spaces, of logarithmic type. \\
Then, we will need to introduce a paradifferential calculus depending on some parameter $\g\geq1$: the main ideas
are the same of the classical version, but the presence of the parameter allows us to perform a more refined analysis.
This will play a crucial role to get our result. \\
After this, we will consider functions with low regularity modulus of continuity. In particular, we will
focus on log-Zygmund and log-Lipschitz functions: dyadic decomposition allows us to get some of their properties. Moreover,
we will analyse the convolution of a log-Zygmund function by a smoothing kernel. \\
Finally, taking advantage of paradifferential calculus with parameters, we will consider general symbols having low regularity in time and space variables. Under suitable hypothesis on such a symbol, we will also
get positivity estimates for the associated operator.

\subsection{Littlewood-Paley decomposition} \label{ss:L-P}

Let us first define the so called ``Littlewood-Paley decomposition'', based on a non-homogeneous dyadic partition of unity with
respect to the Fourier space variable. We refer to \cite{B-C-D}, \cite{Bony} and \cite{M-2008}
for the details.

So, fix a smooth radial function
$\chi$ supported in the ball $B(0,2),$ 
equal to $1$ in a neighborhood of $B(0,1)$
and such that $r\mapsto\chi(r\,e)$ is nonincreasing
over $\R_+$ for all unitary vectors $e\in\R^N$. Set also
$\varphi\left(\xi\right)=\chi\left(\xi\right)-\chi\left(2\xi\right)$.

For convenience, we immediately introduce the following notation:
$$
\chi_j(\xi)\,:=\,\chi(2^{-j}\xi)\qquad\mbox{ and }\qquad \vphi_j(\xi)\,:=\,\vphi(2^{-j}\xi)\,.
$$
We will indifferently use it or the previous one.

\smallbreak
The dyadic blocks $(\Delta_j)_{j\in\Z}$
 are defined by\footnote{Throughout we agree  that  $f(D)$ stands for 
the pseudo-differential operator $u\mapsto\mc{F}^{-1}(f\,\mc{F}u)$.} 
$$
\Delta_j:=0\ \hbox{ if }\ j\leq-1,\quad\Delta_{0}:=\chi(D)\quad\hbox{and}\quad
\Delta_j:=\varphi(2^{-j}D)\ \text{ if }\  j\geq1.
$$
We  also introduce the following low frequency cut-off operator:
$$
S_j\,:=\,\chi(2^{-j}D)\,=\,\sum_{k\leq j}\Delta_{k}\quad\text{for}\quad j\geq0.
$$

The following classical properties will be freely used throughout the paper:
\begin{itemize}
\item for any $u\in\mc{S}',$ the equality $u=\sum_{j}\Delta_ju$ holds true in $\mc{S}'$;
\item for all $u$ and $v$ in $\mc{S}'$,
the sequence $\left(S_{j-3}u\,\,\Delta_jv\right)_{j\in\N}$ is spectrally supported in dyadic annuli.
\end{itemize}

Let us also mention a fundamental result, which explains, by the so-called \emph{Bernstein's inequalities},
the way derivatives act on spectrally localized functions.
  \begin{lemma} \label{l:bern}
Let  $0<r<R$.   A
constant $C$ exists so that, for any multi-index $\alpha\in\N^N$, any couple $(p,q)$ 
in $[1,+\infty]^2$ with  $p\leq q$ 
and any function $u\in L^p$,  we  have, for all $\lambda>0$,
$$
\displaylines{
{\rm supp}\, \widehat u \subset   B(0,\lambda R)\quad
\Longrightarrow\quad
\|\d^\alpha u\|_{L^q}\, \leq\,
 C^{|\alpha|+1}\,\lambda^{|\alpha|+N\left(\frac{1}{p}-\frac{1}{q}\right)}\,\|u\|_{L^p}\;;\cr
{\rm supp}\, \widehat u \subset \{\xi\in\R^N\,|\, r\lambda\leq|\xi|\leq R\lambda\}
\quad\Longrightarrow\quad C^{-|\alpha|-1}\lambda^{|\alpha|}\|u\|_{L^p}\,
\leq\,\|\d^\alpha u\|_{L^p}\,\leq\,
C^{|\alpha|+1}  \lambda^{|\alpha|}\|u\|_{L^p}\,.
}$$
\end{lemma}   

Let us recall the characterization of (classical) Sobolev spaces via dyadic decomposition:
for all $s\in\mbb{R}$ there exists a constant $C_s>0$ such that
\begin{equation} \label{est:dyad-Sob}
 \frac{1}{C_s}\,\,\sum^{+\infty}_{\nu=0}2^{2\,\nu\, s}\,\|u_\nu\|^2_{L^2}\;\leq\;\|u\|^2_{H^s}\;\leq\;
C_s\,\,\sum^{+\infty}_{\nu=0}2^{2\,\nu\, s}\,\|u_\nu\|^2_{L^2}\,,
\end{equation}
where we have set $u_\nu:=\Delta_\nu u$.

So, the $H^s$ norm of a tempered distribution is the same as the $\ell^2$ norm of the sequence
$\left(2^{s\nu}\,\left\|\Delta_\nu u\right\|_{L^2}\right)_{\nu\in\N}$. Now, one may ask what we get if, in the sequence, we put
weights different to the exponential term $2^{s\nu}$. Before answering this question, we introduce some definitions. For the
details of the presentation, we refer also to \cite{C-M}.

Let us set $\Pi(D)\,:=\,\log(2+|D|)$, i.e. its symbol is $\pi(\xi)\,:=\,\log(2+|\xi|)$.
\begin{defin} \label{d:log-H^s}
 For all $\alpha\in\R$, we define the space $H^{s+\alpha\log}$ as the space $\Pi^{-\alpha}H^s$, i.e.
$$
f\,\in\,H^{s+\alpha\log}\quad\lra\quad\Pi^\alpha f\,\in\,H^s\quad\lra\quad
\pi^\alpha(\xi)\left(1+|\xi|^2\right)^{s/2}\what{f}(\xi)\,\in\,L^2\,.
$$
\end{defin}

From the definition, it's obvious that the following inclusions hold true for any $s_1>s_2$ and any $\alpha_1>\alpha_2>0$:
$$
H^{s_1+\alpha_1\log}\;\hra\;H^{s_1+\alpha_2\log}\;\hra\;H^{s_1}\;\hra\;
H^{s_1-\alpha_2\log}\;\hra\;H^{s_1-\alpha_1\log}\;\hra\;H^{s_2}\,.
$$

We have the following dyadic characterization of these spaces (see \cite[Prop. 4.1.11]{M-2008}).
\begin{prop} \label{p:log-H}
 Let $s$, $\alpha\,\in\R$. A tempered distribution $u$ belongs to the space $H^{s+\alpha\log}$ if and only if:
\begin{itemize}
 \item[(i)] for all $k\in\N$, $\Delta_ku\in L^2(\R^N)$;
\item[(ii)] set $\,\delta_k\,:=\,2^{ks}\,(1+k)^\alpha\,\|\Delta_ku\|_{L^2}$ for all $k\in\N$, the sequence
$\left(\delta_k\right)_k$ belongs to $\ell^2(\N)$.
\end{itemize}
Moreover, $\|u\|_{H^{s+\alpha\log}}\,\sim\,\left\|\left(\delta_k\right)_k\right\|_{\ell^2}$.
\end{prop}
Hence, this proposition generalizes property \eqref{est:dyad-Sob}.

This new class of Sobolev spaces, which are in a certain sense of logarithmic type, will come into play in our
analysis. As a matter of fact, operators associated to log-Zygmund or log-Lipschitz symbols
give a logarithmic loss of derivatives.
We will clarify in a while what we have just said; first of all, we
need to introduce a new version of paradifferential calculus, depending on a parameter $\g\geq1$.

\subsection{Paradifferential calculus with parameters}

Let us present here the paradifferential calculus depending on some parameter $\g$. One can find a complete
and detailed treatement in Appendix B of \cite{M-Z} (see also \cite{M-1986}).

Fix $\gamma\geq1$ and take a cut-off function $\psi\in\mc{C}^\infty(\R^N\times\R^N)$ which verifies the following properties:
\begin{itemize}
 \item there exist $0<\veps_1<\veps_2<1$ such that
$$
\psi(\eta,\xi)\,=\,\left\{\begin{array}{lcl}
                           1 & \mbox{for} & |\eta|\leq\veps_1\left(\gamma+|\xi|\right) \\ [1ex]
			   0 & \mbox{for} & |\eta|\geq\veps_2\left(\gamma+|\xi|\right)\,;
                          \end{array}
\right.
$$
\item for all $(\beta,\alpha)\in\N^N\times\N^N$, there exists a constant $C_{\beta,\alpha}$ such that
$$
\left|\d^\beta_\eta\d^\alpha_\xi\psi(\eta,\xi)\right|\,\leq\,C_{\beta,\alpha}\left(\gamma+|\xi|\right)^{-|\alpha|-|\beta|}\,.
$$
\end{itemize}

\begin{rem} \label{r:indep_gamma}
 We remark that $\veps_1$, $\veps_2$ and the different $C_{\beta,\alpha}$ occurring in the previous estimates
must not depend on $\g$.
\end{rem}

For instance, if $\gamma=1$, one can take
$$
\psi(\eta,\xi)\,\equiv\,\psi_{-3}(\eta,\xi)\,:=\,\sum_{k=0}^{+\infty}\chi_{k-3}(\eta)\,\vphi_k(\xi)\,,
$$
where $\chi$ and $\vphi$ are the localization (in phase space) functions associated to a Littlewood-Paley decomposition,
see \cite[Ex. 5.1.5]{M-2008}.
Similarly, if $\gamma>1$ it is possible to find a suitable integer $\mu\geq0$ such that
\begin{equation} \label{pd_eq:pp_symb}
\psi(\eta,\xi)\,\equiv\,\psi_\mu(\eta,\xi)\,:=\,\chi_{\mu}(\eta)\,\chi_{\mu+2}(\xi)\,+\,\sum_{k=\mu+3}^{+\infty}\chi_{k-3}(\eta)\,\vphi_k(\xi)
\end{equation}
is a function with the just described properties.

\smallskip
Define now
$$
G^\psi(x,\xi)\,:=\,\left(\mc{F}^{-1}_\eta\psi\right)(x,\xi)\,,
$$
the inverse Fourier transform of $\psi$ with respect to the variable $\eta$.
\begin{lemma} \label{l:G}
 For all $(\beta,\alpha)\in\N^N\times\N^N$, there exist constants $C_{\beta,\alpha}$, independent of $\g$, such that:
\begin{eqnarray}
 \left\|\d^\beta_x\d^\alpha_\xi G^\psi(\cdot,\xi)\right\|_{L^1(\R^N_x)} & \leq &
C_{\beta,\alpha}\left(\gamma+|\xi|\right)^{-|\alpha|+|\beta|}\,, \label{pd_est:G_1} \\
\left\||\cdot|\,\log\left(2+\frac{1}{|\cdot|}\right)\,\d^\beta_x\d^\alpha_\xi G^\psi(\cdot,\xi)\right\|_{L^1(\R^N_x)} & \leq &
C_{\beta,\alpha}\left(\gamma+|\xi|\right)^{-|\alpha|+|\beta|-1}\,\log(1+\gamma+|\xi|)\,. \label{pd_est:G_2}
\end{eqnarray}
\end{lemma}

\begin{proof}
See \cite[Lemma 5.1.7]{M-2008}. 
\end{proof}

Thanks to $G^\psi$, we can smooth out a symbol $a$ in the $x$ variable and then define the paradifferential operator associated to $a$
as the pseudodifferential operator related to this smooth function. We set the classical symbol associated to $a$ to be
$$
\sigma_a(x,\xi)\,:=\,\left(\,\psi(D_x,\xi)\,a\,\right)(x,\xi)\,=\,\left(G^\psi(\cdot,\xi)\,*_x\,a(\cdot,\xi)\right)(x)\,,
$$
and then the paradifferential operator associated to $a$:
$$T_a\,:=\,\sigma_a(x,D_x)\,,$$
where we have omitted $\psi$ because the definition is independent of it, up to lower order terms.
\begin{rem} \label{r:p-prod}
 Let us note that if $a=a(x)\in L^\infty$ and if we take the cut-off function $\psi_{-3}$, then $T_a$ is actually the usual
paraproduct operator. If we take $\psi_\mu$ as defined in \eqref{pd_eq:pp_symb}, instead, we get a paraproduct operator
which starts from high enough frequencies, which will be indicated with $T^\mu_a$(see \cite [Par. 3.3]{C-M}).
\end{rem}

Let us point out that we can also define a $\gamma$-dyadic decomposition. First of all, we set
$$
\Lambda(\xi,\g)\,:=\,\left(\g^2\,+\,|\xi|^2\right)^{1/2}\,.
$$
Then, taken the usual smooth function $\chi$ associated to a  Littlewood-Paley decomposition, we can define
$$
\chi_\nu(\xi,\g)\,:=\,\chi\left(2^{-\nu}\Lambda(\xi,\g)\right)\,,\quad
S^\g_\nu\,:=\,\chi_\nu(D_x,\g)\,,\quad
\Delta^\g_\nu\,:=\,S^\g_{\nu+1}-S^\g_\nu\,.
$$
The usual properties of the support of the localization functions still hold, and
for all fixed $\g\geq1$ and all tempered distributions $u$, we have
$$
u\,=\,\sum_{\nu=0}^{+\infty}\,\Delta^\g_\nu\,u\qquad\mbox{in }\;\;\mc{S}'\,.
$$
Moreover, with natural modifications in definitions, we can introduce the space $H^{s+\alpha\log}_\gamma$ as the set of
tempered distributions for which
$$
\left\|u\right\|^2_{H^{s+\alpha\log}_\g}\,:=\,\int_{\R^N_\xi}\Lambda^{2s}(\xi,\g)\,\log^{2\alpha}(1+\g+|\xi|)\,
\left|\what{u}(\xi)\right|^2\,d\xi\;\;<\;+\infty\,.
$$
For the details see \cite[Appendix B]{M-Z}. What is important to retain is that,
once we fix $\g\geq1$
(for example, to obtain positivity of paradifferential operators involved in our computations), the whole previous construction is
equivalent to the classical one; in particular,
the space $H^{s+\alpha\log}_\gamma$ coincides with $H^{s+\alpha\log}$, the respective norms are equivalent and the
characterization given by Proposition \ref{p:log-H} still holds true.

\subsection{On log-Lipschitz and log-Zygmund functions}

Let us now give the rigorous definitions of the moduli of continuity we are dealing with, and state some of
their properties.
\begin{defin} \label{d:LL}
 A function $f\in L^\infty(\R^N)$ is said to be log-Lipschitz, and we write $f\in LL(\R^N)$, if the quantity
$$
|f|_{LL}\,:=\,\sup_{x,y\in\R^N,\,|y|<1}
\left(\frac{\left|f(x+y)\,-\,f(x)\right|}{|y|\,\log\left(1\,+\,\frac{1}{|y|}\right)}\right)\,<\,+\infty\,.
$$
We define $\|f\|_{LL}\,:=\,\|f\|_{L^\infty}\,+\,|f|_{LL}$.
\end{defin}
Let us define also the space of log-Zygmund functions. We will give the general definition in $\R^N$, even if the
one dimensional case will be the only relevant one for our purposes.
\begin{defin} \label{d:LZ}
 A function $g\in L^\infty(\R^N)$ is said to be log-Zygmund, and we write $g\in LZ(\R^N)$, if the quantity
$$
|g|_{LZ}\,:=\,\sup_{x,y\in\R^N,\,|y|<1}
\left(\frac{\left|g(x+y)\,+\,g(x-y)\,-\,2\,g(x)\right|}{|y|\,\log\left(1\,+\,\frac{1}{|y|}\right)}\right)\,<\,+\infty\,.
$$
We define $\|g\|_{LZ}\,:=\,\|g\|_{L^\infty}\,+\,|g|_{LZ}$.
\end{defin}

\begin{rem} \label{r:log-g}
Let us immediately point out that, by monotonicity of logarithmic function, we can replace the factor
$\log\left(1+1/|y|\right)$ in previous definitions with $\log\left(1+\g+1/|y|\right)$, for all parameters $\g\geq1$.
As paradifferential calculus with parameters will play a fundamental role in our computations, it's convenient to perform
such a change, and to do so also in hypothesis \eqref{h:LZ_t} and \eqref{h:LL_x} of section \ref{s:results}.
\end{rem}

Let us give a characterization of the space $LZ$. Recall that the space of Zygmund functions is actually
$B^1_{\infty,\infty}$ (see e.g. \cite{Ch1995}): arguing in the same way, one can prove the next proposition.
\begin{prop} \label{p:LZ}
 The space $LZ(\R^N)$ coincides with the logarithmic Besov space $B^{1-\log}_{\infty,\infty}$, i.e. the space of tempered
distributions $u$ such that
\begin{equation} \label{est:LZ}
 \sup_{\nu\geq0}\biggl(2^{\nu}\,(\nu+1)^{-1}\,\left\|\Delta_\nu u\right\|_{L^\infty}\biggr)\,<\,+\infty\,.
\end{equation}
\end{prop}

\begin{proof}
 \begin{itemize}
  \item[(i)] Let us first consider a $u\in B^{1-\log}_{\infty,\infty}$ and take $x$ and $y\,\in\R^N$, with
$|y|<1$. For all fixed $n\in\N$ we can write:
\begin{eqnarray*}
u(x+y)+u(x-y)-2u(x) & = & \sum_{k<n}\left(\Delta_ku(x+y)+\Delta_ku(x-y)-2\Delta_ku(x)\right) \\
& & \qquad\qquad+\,\sum_{k\geq n}\left(\Delta_ku(x+y)+\Delta_ku(x-y)-2\Delta_ku(x)\right)\,.
\end{eqnarray*}
First, we take advantage of the Taylor's formula up to second order to handle the former terms; then, we use property
\eqref{est:LZ}. Hence we get
\begin{eqnarray*}
 \left|u(x+y)+u(x-y)-2u(x)\right| & \leq & C\,|y|^2\sum_{k<n}\left\|\nabla^2\Delta_ku\right\|_{L^\infty}\,+\,
4\,\sum_{k\geq n}\left\|\Delta_ku\right\|_{L^\infty} \\
& \leq & C\left(|y|^2\sum_{k<n}2^k\,(k+1)\,+\,\sum_{k\geq n}2^{-k}(k+1)\right) \\
& \leq & C\,(n+1)\left(|y|^2\,2^n\,+\,2^{-n}\right)\,.
\end{eqnarray*}
Now, as $|y|<1$, the choice $n=1+\left[\log_2\left(1/|y|\right)\right]$ (where with $[\sigma]$ we mean the greatest positive
integer less than or equal to $\sigma$) completes the proof of the first part.
\item[(ii)] Now, given a log-Zygmund function $u$, we want to estimate the $L^\infty$ norm of its localized part $\Delta_ku$.

Let us recall that applying the operator $\Delta_k$ is the same of the convolution with the inverse Fourier transform of
the function $\vphi(2^{-k}\cdot)$, which we call $h_k(x)=2^{kN}h(2^k\cdot)$, where we set
$h=\mc{F}^{-1}_\xi(\vphi)$. As $\vphi$ is an even function, so does $h$; moreover we have
$$
\int h(z)\,dz\,=\,\int\mc{F}^{-1}_\xi(\vphi)(z)\,dz\,=\,\vphi(\xi)_{|\xi=0}\,=\,0\,.
$$
Therefore, we can write:
$$
\Delta_ku(x)\,=\,2^{kN-1}\int h(2^ky)\left(u(x+y)+u(x-y)-2u(x)\right)\,dy\,,
$$
and noting that $\sigma\,\mapsto\,\sigma\,\log\left(1+\g+1/\sigma\right)$ is increasing completes the proof of the second part.
 \end{itemize}
\end{proof}

From definitions \ref{d:LL} and \ref{d:LZ}, it's obvious that $LL(\R^N)\,\hra\,LZ(\R^N)$: Proposition 3.3 of \cite{C-L}
explains this property in terms of dyadic decomposition.
\begin{prop} \label{p:dyadic-LL}
 There exists a constant $C$ such that, for all $a\in LL(\R^N)$ and all integers $k\geq0$, we have
\begin{equation}
 \left\|\Delta_k a\right\|_{L^\infty}\,\leq\,C\,(k+1)\,2^{-k}\,\|a\|_{LL}\,. \label{est:a_k}
\end{equation}
 Moreover, for all $k\in\N$ we have
\begin{eqnarray}
 \left\|a\,-\,S_k a\right\|_{L^\infty} & \leq & C\,(k+1)\,2^{-k}\,\|a\|_{LL} \label{est:a-S_ka} \\
\left\|S_k a\right\|_{\mc{C}^{0,1}} & \leq & C\,(k+1)\,\|a\|_{LL}\,. \label{est:S_ka}
\end{eqnarray}
\end{prop}

\begin{rem} \label{r:LL}
 Note that, again from Proposition 3.3 of \cite{C-L}, property \eqref{est:S_ka} is a characterization of the space $LL$.
\end{rem}

Using dyadic characterization of the space $LZ$ and arguing as for proving Proposition \ref{p:LZ}, we can establish
the following property. For our purposes, it's enough to consider a log-Zygmund function $a$ depending only on the
time variable $t$, but the same reasoning holds true also in higher dimensions.
\begin{lemma} \label{l:LZ}
 For all $a\in LZ(\R)$, there exists a constant $C$, depending only on the $LZ$ norm of $a$,
such that, for all $\gamma\geq1$ and all $0<|\tau|<1$ one has
\begin{equation} \label{est:lip-LZ}
 \sup_{t\in\R}\left|a(t+\tau)-a(t)\right|\,\leq\,C\,|\tau|\,\log^2\!\left(1+\gamma+\frac{1}{|\tau|}\right).
\end{equation}
\end{lemma}

\begin{proof}
As done in proving Proposition \ref{p:LZ}, for all $n\in\N$ we can write
$$
a(t+\tau)-a(t)\,=\,\sum_{k<n}\left(\Delta_ka(t+\tau)-\Delta_ka(t)\right)\,+\,
\sum_{k\geq n}\left(\Delta_ka(t+\tau)-\Delta_ka(t)\right)\,,
$$
where, obviously, the localization in frequencies is done with respect to the time variable. For the former terms we
use the mean value theorem, while for the latter ones we use characterization \eqref{est:a_k}; hence, we get
\begin{eqnarray*}
\left|a(t+\tau)-a(t)\right| & \leq & \sum_{k<n}\left\|\frac{d}{dt}\Delta_ka\right\|_{L^\infty}|\tau|\,+\,
2\sum_{k\geq n}\left\|\Delta_ka\right\|_{L^\infty} \\
& \leq & C\left(n^2\,|\tau|\,+\,\sum_{k\geq n}2^{-k}k\right).
\end{eqnarray*}
The series in the right-hand side of the previous inequality can be bounded, up to a multiplicative constant, by
$2^{-n}n$; therefore
$$
\left|a(t+\tau)-a(t)\right|\,\leq\,C\,n\left(n\,|\tau|\,+\,2^{-n}\right)\,,
$$
and the choice $n=1+[\log_2(1/|\tau|)]$ completes the proof.
\end{proof}

Now, given a log-Zygmund function $a(t)$, we can regularize it by convolution. So, take an
even function $\rho\in\mc{C}^\infty_0(\mbb{R}_t)$, $0\leq\rho\leq1$, whose support is contained in the interval $[-1,1]$ and
such that $\int\rho(t)dt=1$, and define the mollifier kernel
$$
\rho_\veps(t)\,:=\,\frac{1}{\veps}\,\,\rho\!\left(\frac{t}{\veps}\right)\qquad\qquad\forall\,\veps\in\,]0,1]\,.
$$
We smooth out the function $a$ setting, for all $\veps\in\,]0,1]$,
\begin{equation} \label{eq:a^e}
a_\veps(t)\,:=\,\left(\rho_\veps\,*\,a\right)(t)\,=\,\int_{\mbb{R}_s}\rho_{\veps}(t-s)\,a(s)\,ds\,.
\end{equation}
The following proposition holds true.
\begin{prop} \label{p:LZ-reg}
 Let $a$ be a log-Zygmund function. There exist constants $C$ such that, for all $\g\geq1$, one has
\begin{eqnarray}
\left|a_\veps(t)-a(t)\right| & \leq & C\,\|a\|_{LZ}\,\,\veps\,\log\left(1+\gamma+\frac{1}{\veps}\right) \label{est:a^e-a} \\
\left|\d_ta_\veps(t)\right| & \leq & C\,\|a\|_{LZ}\,\log^2\!\left(1+\gamma+\frac{1}{\veps}\right) \label{est:d_t-a^e} \\
\left|\d^2_ta_\veps(t)\right| & \leq & C\,\|a\|_{LZ}\,\,\frac{1}{\veps}\,
\log\!\left(1+\gamma+\frac{1}{\veps}\right) \label{est:d_tt-a^e}\,.
\end{eqnarray}
\end{prop}

\begin{proof}
 For first and third inequalities, the proof is the same as in \cite{C-DS}.
We have to pay attention only to \eqref{est:d_t-a^e}. As $\rho'$ has null integral, the relation
$$
\d_ta_\veps(t)\,=\,\frac{1}{\veps^2}\int_{|s|\leq\veps}\rho'\left(\frac{s}{\veps}\right)\,
\left(a(t-s)-a(t)\right)ds
$$
holds, and hence, taking advantage of \eqref{est:lip-LZ}, it implies
$$
\left|\d_ta_\veps(t)\right|\,\leq\,\frac{C}{\veps^2}\,\int_{|s|\leq\veps}
\left|\rho'\left(\frac{s}{\veps}\right)\right|\,|s|\,\log^2\!\left(1+\gamma+\frac{1}{|s|}\right)ds\,.
$$
Observing that the function $\sigma\mapsto\sigma\log^2(1+\gamma+1/\sigma)$ is increasing in the interval $[0,1]$, and so
does in $[0,\veps]$, allows us to complete the proof.
\end{proof}

\subsection{Low regularity symbols and calculus}

For the analysis of our strictly hyperbolic problem, it's important to pass from $LZ_t-LL_x$ functions to more general
symbols in variables $(t,x,\xi)$ which have this same regularity in $t$ and $x$. 

We want to investigate properties of these symbols and of the associated operators. For reasons which will appear
clear in the sequel, we will have to take advantage not of the classical paradifferential calculus, but of the calculus
with parameters. Therefore, we will allow the symbols to depend also on a parameter $\g\geq1$.

We point out that in our calculus the time can be treated as an additional parameter, while $\xi$ represents, as usual, the
dual variable to $x$.

Let us start with a definition (see also \cite[Def. B.4]{M-Z}).
\begin{defin}\label{d:symbols}
 A \emph{symbol of order $m+\delta\log$} (for $m,\delta\,\in\,\R$) is a function $a(t,x,\xi,\g)$ which is locally bounded on
$[0,T_0]\times\R^N\times\R^N\times[1,+\infty[\,$, of class $\mc{C}^\infty$ with respect to $\xi$ and which satisfy the following
property: for all $\alpha\in\N^N$, there exists a constant $C_\alpha>0$ such that
\begin{equation} \label{est:def_symb}
 \left|\d^\alpha_\xi a(t,x,\xi,\g)\right|\,\leq\,C_\alpha\,\left(\g+|\xi|\right)^{m-|\alpha|}\,
\log^\delta\left(1+\g+|\xi|\right)
\end{equation}
for all $(t,x,\xi,\g)$.
\end{defin}

So, take a symbol $a(t,x,\xi,\g)$ of order $m\geq0$, which is log-Zygmund in $t$ and log-Lipschitz in $x$,
uniformly with respect to the other variables.
Now we smooth $a$ out with respect to time, as done in \eqref{eq:a^e}.
Next lemma provides us some estimates on classical symbols
associated to $a_\veps$ and its time derivatives.
\begin{lemma} \label{l:lzll-symb}
 The classical symbols associated to $a_\veps$ and its time derivatives satisfy:
\begin{eqnarray*}
 \left|\d^\alpha_\xi\sigma_{a_\veps}\right| & \leq & C_\alpha\left(\gamma+|\xi|\right)^{m-|\alpha|} \\
\left|\d^\beta_x\d^\alpha_\xi\sigma_{a_\veps}\right| & \leq & C_{\beta,\alpha}\left(\gamma+|\xi|\right)^{m-|\alpha|+|\beta|-1}\,
\log\left(1+\gamma+|\xi|\right) \\
\left|\d^\alpha_\xi\sigma_{\d_ta_\veps}\right| & \leq & C_\alpha\left(\gamma+|\xi|\right)^{m-|\alpha|}\,
\log^2\!\!\left(1+\gamma+\frac{1}{\veps}\right) \\
\left|\d^\beta_x\d^\alpha_\xi\sigma_{\d_ta_\veps}\right| & \leq & C_{\beta,\alpha}\,\,\left(\gamma+|\xi|\right)^{m-|\alpha|+|\beta|-1}\,
\log\left(1+\gamma+|\xi|\right)\,\frac{1}{\veps} \\
\left|\d^\alpha_\xi\sigma_{\d^2_ta_\veps}\right| & \leq & C_\alpha\left(\gamma+|\xi|\right)^{m-|\alpha|}\,
\log\!\!\left(1+\gamma+\frac{1}{\veps}\right)\,\frac{1}{\veps} \\
\left|\d^\beta_x\d^\alpha_\xi\sigma_{\d^2_ta_\veps}\right| & \leq & C_{\beta,\alpha}\,\left(\gamma+|\xi|\right)^{m-|\alpha|+|\beta|-1}\,
\log\left(1+\gamma+|\xi|\right)\,\frac{1}{\veps^2}\,\,,
\end{eqnarray*}
where all the constants which occur here don't depend on the parameter $\g$.
\end{lemma}

\begin{proof}
 The first inequality is a quite easy computation.

For the second one, we have to observe that
$$
\int\d_iG(x-y,\xi)dx\,=\,\int\d_iG(z,\xi)dz\,=\,\int\mc{F}^{-1}_\eta\left(\eta_i\,\psi(\eta,\xi)\right)dz\,=\,
\left(\eta_i\,\psi(\eta,\xi)\right)_{|\eta=0}\,=\,0\,.
$$
So, we have
$$
\d_i\sigma_{a_\veps}\,=\,\int\d_iG(y,\xi)\left(a_\veps(t,x-y,\xi,\g)-a_\veps(t,x,\xi,\g)\right)dy\,,
$$
and from this, remembering lemma \ref{l:G}, we get the final control.

The third estimate immediately follows from the hypothesis on $a$ and from \eqref{est:d_t-a^e}.

Moreover, in the case of space derivatives, we can take advantage once again of the fact that $\d_iG$ has null integral:
\begin{eqnarray*}
 \d_i\sigma_{\d_ta_\veps} & = & \int\d_iG(x-y,\xi)\,\d_ta_\veps(t,y,\xi,\g)\,dy \\
& = & \int_{\R_s}\frac{1}{\veps^2}\,\,\rho'\left(\frac{t-s}{\veps}\right)
\left(\int_{\R^N_y}\d_iG(y,\xi)\left(a(s,x-y,\xi,\g)-a(s,x,\xi,\g)\right)dy\right)ds\,.
\end{eqnarray*}
Hence, the estimate follows from the log-Lipschitz continuity hypothesis and from inequality \eqref{pd_est:G_2}.

About the $\d^2_ta_\veps$ term, the first estimate comes from \eqref{est:d_tt-a^e}, while for the second one we argue as before:
\begin{eqnarray*}
 \d_i\sigma_{\d^2_ta_\veps} & = & \int\d_iG(x-y,\xi)\,\d^2_ta_\veps(t,y,\xi,\g)\,dy \\
& = & \int_{\R^N_y}\d_iG(x-y,\xi)\,\frac{1}{\veps^3}
\left(\int_{\R_s}\rho''\left(\frac{t-s}{\veps}\right)\left(a(s,y,\xi,\g)-a(s,x,\xi,\g)\right)ds\right)dy \\
& = & \frac{1}{\veps^3}\int_{\R_s}\rho''\left(\frac{t-s}{\veps}\right)
\left(\int_{\R^N_y}\d_iG(y,\xi)\left(a(s,x-y,\xi,\g)-a(s,x,\xi,\g)\right)dy\right)ds\,,
\end{eqnarray*}
and the thesis follows again from log-Lipschitz continuity and from \eqref{pd_est:G_2}.
\end{proof}

\begin{rem} \label{r:symb_calc}
Note that first and second inequalities are satisfied also by the symbol $a$ (not smoothed in time).
\end{rem}

\medbreak
Now let us quote some basic facts on symbolic calculus, which follow from previous lemma. Before doing this, we
recall a definition (see also \cite[Def. B.8]{M-Z}).
\begin{defin} \label{d:op_order}
 We say that an operator $P$ is of order $\,m+\delta\log\,$ if, for every $(s,\alpha)\in\R^2$ and every $\g\geq1$,
$P$ maps $H^{s+\alpha\log}_\g$ into $H^{(s-m)+(\alpha-\delta)\log}_\g$ continuously.
\end{defin}

\begin{prop} \label{p:symb-calc}
 \begin{itemize}
 \item[(i)] Let $a$ be a symbol of order $m$ which is $L^\infty$ in the $x$ variable.
Then $T_a$ is of order $m$, i.e. it maps $H^{s+\alpha\log}_\g$ into $H^{s-m+\alpha\log}_\g$.
  \item[(ii)] Let us take two symbols $a$, $b$ of order $m$ and $m'$ respectively. Suppose that $a$, $b$ are $LL$ in the $x$ variable.
The composition  of the associated operators can be approximated by the symbol associated to the product of
these symbols, up to a remainder term:
$$
T_{a}\,\circ\,T_b\,=\,T_{ab}\,+\,R\,.
$$
$R$ has order $m+m'-1-\log$: it maps $H^{s+\alpha\log}_\g$ into $H^{s-m-m'+1+(\alpha+1)\log}_\g$.
\item[(iii)] Let $a$ be a symbol of order $m$ which is $LL$ in the $x$ variable. The adjoint (over $L^2$) operator of $T_a$ is,
up to a remainder operator, $T_{\oline{a}}$. The remainder operator is of order $m-1+\log$, and it maps
$H^{s+\alpha\log}_\g$ into $H^{s-m+1+(\alpha+1)\log}_\g$.
 \end{itemize}
\end{prop}

Let us end this subsection stating a basic positivity estimate.
In this situation, paradifferential calculus with parameters comes into play.
\begin{prop} \label{p:pos}
 Let $a(t,x,\xi,\g)$ be a symbol of order $m$, which is $LL$ in the $x$ variable and such that
$$
\Re\,a(t,x,\xi,\g)\,\geq\,\lambda_0\,\left(\gamma+|\xi|\right)^m\,.
$$

Then, there exists a constant $\lambda_1$, depending only on $|a|_{LL}$ and on $\lambda_0$ (so not on $\g$), such that,
for $\gamma$ large enough, one has
$$
\Re\!\left(T_au,u\right)_{L^2(\R^N_x)}\,\geq\,\lambda_1\,\|u\|^2_{H^{m/2}_\gamma}\,.
$$
\end{prop}

\begin{proof}
 The result is an immediate consequence of Theorem B.18 of \cite{M-Z}, together with Proposition \ref{p:symb-calc} about the
remainder for composition and adjoint operators.
\end{proof}

\begin{rem} \label{r:pos}
 Let us note the following fact, which comes again from Theorem B.18 of \cite{M-Z}.
If the positive symbol $a$ has low regularity in time and we smooth it out by convolution with respect to this variable,
we obtain a family $\left(a_\veps\right)_\veps$ of positive symbols, with same constant $\lambda_0$.
Now, all the paradifferential operators associated to these symbols will be positive operators, uniformly in $\veps$: i.e. the
constant $\lambda_1$ of the previous inequality can be chosen independently of $\veps$.
\end{rem}

Let us observe that the previous proposition generalizes Corollary 3.12 of \cite{C-M} (stated for the paraproduct by a
positive $LL$ function) to the more general case of a paradifferential operator related to a strictly positive symbol of order $m$.

Finally, thanks to Theorem B.17 of \cite{M-Z} about the remainder operator for the adjoint, we have the following corollary,
which turns out to be fundamental in our energy estimates.

\begin{coroll} \label{c:pos}
 Let $a$ be a positive symbol of order $1$ and suppose that $a$ is $LL$ in the $x$ variable.

Then there exists $\gamma\geq1$, depending only on the symbol $a$, such that
$$
\left\|T_au\right\|_{L^2}\,\sim\,\left\|u\right\|_{H^1_\g}
$$ 
for all $u\in H^1_\g(\R^N)$.
\end{coroll}

\section{Proof of the energy estimate for $L$}

Finally, we are able to tackle the proof of Theorem \ref{t:en-est}. We argue in a standard way: we prove
energy estimates, for some suitable energy associated to a solution of equation \eqref{eq:op}.

The key idea to the proof is to split the total energy into localized components $e_\nu$, each one of them associated
to the dyadic block $\Delta_\nu u$, and then to put all these pieces together (see also \cite{C-L} and \cite{C-DS}). Let us
see the proof into details.

\subsection{Approximate and total energy}

Let us first regularize coefficients $a_{ij}$ in the time variable by convolution, as done in \eqref{eq:a^e},
and let us define the $0$-th order symbol
$$
\alpha_\veps(t,x,\xi)\,:=\,\left(\gamma^2+|\xi|^2\right)^{-1/2}\,
\left(\gamma^2\,+\,\sum_{i,j}a_{ij, \veps}(t,x)\,\xi_i\,\xi_j\right)^{1/2}\,. 
$$
We take $\veps\,=\,2^{-\nu}$ (see also \cite{C-L} and \cite{C-DS}), and (for notation convenience) we will miss out the $\veps$.

Before going on, let us fix a real number $\gamma\geq1$, 
which will depend only on $\lambda_0$ and on the
$\sup_{i,j}\left|a_{ij}\right|_{LL_x}$, such that (see Corollary \ref{c:pos})
\begin{equation} \label{est:param}
\left\|T_{\alpha^{-1/2}}\,w\right\|_{L^2}\,\geq\,\frac{\lambda_0}{2}\,\left\|w\right\|_{L^2}\qquad\mbox{ and }\qquad
\left\|T_{\alpha^{1/2}(\g^2+|\xi|^2)^{1/2}}\,w\right\|_{L^2}\,\geq\,
\frac{\lambda_0}{2}\,\left\|w\right\|_{H^1}
\end{equation}
for all $w\in H^{\infty}$. 
Let us remark that the choice of $\gamma$ is equivalent to the choice of the parameter $\mu$ in (\ref{pd_eq:pp_symb})
(see Remark \ref{r:p-prod}): hence, from now on we will consider paraproducts starting from this ${\mu}$, according to
definition \eqref{pd_eq:pp_symb}, even if we will omit it in the notations.

Consider in \eqref{eq:op} a function $u\in\mc{C}^2([0,T_0];H^{\infty})$. 
We want to get energy estimates for $u$. 
We rewrite the equation using paraproduct operators by the coefficients $a_{ij}$:
$$
\d^2_tu =\sum_{i,j}\d_i\left(a_{ij}\,\d_ju\right)\,+\, Lu\\
= \sum_{i,j}\d_i\left(T_{a_{ij}} \d_ju\right) \,+\, \wtilde Lu, 
$$
where $\wtilde Lu= Lu + \sum_{i,j}\d_i\left((a_{ij}-T_{a_{ij}})\d_ju\right)$. Let us apply the operator $\Delta_\nu$: we get
\begin{equation} \label{eq:loc}
\d^2_tu_\nu\,=\sum_{i,j}\d_i\left(T_{a_{ij}}\,\d_ju_\nu\right)\,+\,
\sum_{i,j}\d_i\left(\left[\Delta_\nu,T_{a_{ij}}\right]\,\d_ju\right)+\,(\wtilde Lu)_\nu, 
\end{equation}
where $u_\nu=\Delta_\nu u$, $(\wtilde Lu)_\nu=\Delta_\nu (\wtilde Lu)$ and $\left[\Delta_\nu,T_{a_{ij}}\right]$ is the commutator
between $\Delta_\nu$ and the paramultiplication by $a_{ij}$.

Now, we set
\begin{eqnarray*}
 v_\nu(t,x) & := & T_{\alpha^{-1/2}}\,\d_tu_\nu\,-\,T_{\d_t(\alpha^{-1/2})}\,u_\nu \\
w_\nu(t,x) & := & T_{\alpha^{1/2}(\gamma^2+|\xi|^2)^{1/2}}\,u_\nu \\
z_\nu(t,x) & := & u_\nu
\end{eqnarray*}
and we define the approximate energy associated to the $\nu$-th component of $u$ (as done in \cite{C-DS}):
\begin{equation} \label{eq:appr_en}
 e_{\nu}(t)\,:=\,\left\|v_\nu(t)\right\|^2_{L^2}\,+\,\left\|w_\nu(t)\right\|^2_{L^2}\,+\,\left\|z_\nu(t)\right\|^2_{L^2}\,.
\end{equation}

\begin{rem} \label{r:w_nu}
Note that $\left\|w_\nu(t)\right\|^2_{L^2}\,\sim\,\left\|\nabla u_\nu\right\|^2_{L^2}\,\sim\,
2^{2\nu}\,\|u_\nu\|^2_{L^2}$, thanks to hypothesis \eqref{h:hyp} and the choice of the frequency $\mu$ from which the paraproduct
starts, recall also \eqref{est:param}.
\end{rem}

Now, we fix a $\theta\in\,]0,1[\,$, as required in the hypothesis, and we take a $\beta>0$ to be chosen later; we can define the
total energy associated to the solution $u$ to be the quantity
\begin{equation} \label{eq:tot_E}
 E(t)\,:=\,\sum_{\nu\geq0}e^{-2\beta(\nu+1)t}\,2^{-2\nu\theta}\,e_{\nu}(t)\,.
\end{equation}

It's not difficult to prove (see also inequality \eqref{est:d_t-u_nu} below) that there exist constants $C_\theta$ and
$C'_\theta$, depending only on the fixed $\theta$, for which one has:
\begin{eqnarray}
 \left(E(0)\right)^{1/2} & \leq & C_\theta\left(\|\d_tu(0)\|_{H^{-\theta}}\,+\,\|u(0)\|_{H^{-\theta+1}}\right) \label{est:E(0)} \\
\left(E(t)\right)^{1/2} & \geq & C'_\theta\left(\|\d_tu(t)\|_{H^{-\theta-\beta^*t}}\,+\,
\|u(t)\|_{H^{-\theta+1-\beta^*t}}\right)\,, \label{est:E(t)}
\end{eqnarray}
where we have set $\beta^*=\beta\left(\log2\right)^{-1}$.

\subsection{Time derivative of the approximate energy}

Let's find an estimate on the time derivative of the energy. We start analysing each term of \eqref{eq:appr_en}.

\subsubsection{$z_\nu$ term}

For the third term we have:
\begin{equation} \label{eq:d_t-z}
 \frac{d}{dt}\|z_\nu(t)\|^2_{L^2}\,=\,2\,\Re\left(u_\nu\,,\,\d_tu_\nu\right)_{L^2}\,.
\end{equation}
Now, we have to control the term $\d_tu_\nu$: using positivity of operator $T_{\alpha^{-1/2}}$, we get
\begin{equation} \label{est:d_t-u_nu}
 \left\|\d_tu_\nu\right\|_{L^2}\;\leq\;C\,\left\|T_{\alpha^{-1/2}}\d_tu_\nu\right\|_{L^2}
\;\;\leq\;\;C\left(\|v_\nu\|_{L^2}\,+\,\left\|T_{\d_t\left(\alpha^{-1/2}\right)}u_\nu\right\|_{L^2}\right)
\;\leq\;C\left(e_{\nu}\right)^{1/2}\,.
\end{equation}
So, we find the estimate:
\begin{equation} \label{est:d_t-z}
 \frac{d}{dt}\|z_\nu(t)\|^2_{L^2}\,\leq\,C\,e_{\nu}(t)\,.
\end{equation}

\subsubsection{$v_\nu$ term}

Straightforward computations show that
$$
\d_tv_\nu(t,x)\,=\,T_{\alpha^{-1/2}}\d^2_tu_\nu\,-\,T_{\d^2_t(\alpha^{-1/2})}u_\nu\,.
$$

Therefore, keeping in mind relation \eqref{eq:loc}, we get:
\begin{eqnarray} \label{eq:d_t-v}
 \frac{d}{dt}\|v_\nu(t)\|^2_{L^2} & = & 2\,\Re\left(v_\nu\,,\,T_{\alpha^{-1/2}}\left(\wtilde Lu\right)_{\!\!\nu}\,\right)_{L^2}\,-\,2\,\Re\left(v_\nu\,,\,T_{\d^2_t(\alpha^{-1/2})}u_\nu\right)_{L^2} \\
& & \qquad\qquad+\,
2\sum_{i,j}\Re\left(v_\nu\,,\,T_{\alpha^{-1/2}}\d_i\left(T_{a_{ij}}\,\d_ju_\nu\right)\right)_{L^2} \nonumber \\
& & \qquad\qquad\qquad\qquad+\,2\sum_{i,j}\Re\left(v_\nu\,,\,T_{\alpha^{-1/2}}\d_i\left[\Delta_\nu,T_{a_{ij}}\right]\d_ju\right)_{L^2}\,. \nonumber
\end{eqnarray}

Obviously, we have
\begin{equation} \label{est:Lu_nu}
\left|2\,\Re\left(v_\nu\,,\,T_{\alpha^{-1/2}}\left(\wtilde Lu\right)_{\!\!\nu}\,\right)_{L^2}\right|\,\leq\,
C\,\left(e_{\nu}\right)^{1/2}\,\left\|\left(\wtilde Lu\right)_{\!\!\nu}\,\right\|_{L^2}\,,
\end{equation}
while from Lemma \ref{l:lzll-symb} one immediately recovers
\begin{eqnarray}
\left|2\,\Re\left(v_\nu\,,\,T_{\d^2_t(\alpha^{-1/2})}u_\nu\right)_{L^2}\right| & \leq & C\,
\|v_\nu\|_{L^2}\,\log\!\left(1+\gamma+\frac{1}{\veps}\right)\frac{1}{\veps}\,\|u_\nu\|_{L^2} \label{est:zeta} \\
&\leq & C\,(\nu+1)\,e_{\nu}\,, \nonumber
\end{eqnarray}
where we have used the fact that $\varepsilon=2^{-\nu}$.
The other two terms of \eqref{eq:d_t-v} will be treated later.

\subsubsection{$w_\nu$ term}

We now differentiate $w_\nu$ with respect to the time variable: thanks to a broad use of symbolic calculus, we get
the following sequence of equalities:
\begin{eqnarray}
 \frac{d}{dt}\left\|w_\nu\right\|^2_{L^2} & = & 2\,\Re\left(T_{\d_t(\alpha^{1/2})(\gamma^2+|\xi|^2)^{1/2}}u_\nu\,,\,w_\nu\right)_{L^2}
\,+\,2\,\Re\left(T_{\alpha^{1/2}(\gamma^2+|\xi|^2)^{1/2}}\d_tu_\nu\,,\,w_\nu\right)_{L^2} \label{eq:d_t-w} \\
& = & 2\,\Re\left(T_{\alpha(\gamma^2+|\xi|^2)^{1/2}}T_{-\d_t(\alpha^{-1/2})}u_\nu\,,\,w_\nu\right)_{L^2}\,+\,
2\,\Re\left(R_1u_\nu\,,\,w_\nu\right)_{L^2} \nonumber \\
& & +\,2\,\Re\left(T_{\alpha(\gamma^2+|\xi|^2)^{1/2}}T_{\alpha^{-1/2}}\d_tu_\nu\,,\,w_\nu\right)_{L^2}\,+\,
2\,\Re\left(R_2\d_tu_\nu\,,\,w_\nu\right)_{L^2} \nonumber \\
& = & 2\,\Re\left(v_\nu\,,\,T_{\alpha(\gamma^2+|\xi|^2)^{1/2}}w_\nu\right)_{L^2}\,+\,
2\,\Re\left(v_\nu\,,\,R_3w_\nu\right)_{L^2}  \nonumber\\
& & +\,2\,\Re\left(R_1u_\nu\,,\,w_\nu\right)_{L^2}\,+\,2\,\Re\left(R_2\d_tu_\nu\,,\,w_\nu\right)_{L^2} \nonumber \\
& = & 2\,\Re\left(v_\nu\,,\,T_{\alpha^{-1/2}}T_{\alpha^{3/2}(\gamma^2+|\xi|^2)^{1/2}}w_\nu\right)_{L^2}\,+\,
2\,\Re\left(v_\nu\,,\,R_4w_\nu\right)_{L^2} \nonumber \\
& & +\,2\,\Re\left(v_\nu\,,\,R_3w_\nu\right)_{L^2}\,+\,2\,\Re\left(R_1u_\nu\,,\,w_\nu\right)_{L^2}\,+\,
2\,\Re\left(R_2\d_tu_\nu\,,\,w_\nu\right)_{L^2} \nonumber \\
& = & 2\,\Re\left(v_\nu\,,\,T_{\alpha^{-1/2}}T_{\alpha^2(\gamma^2+|\xi|^2)}u_\nu\right)_{L^2} \nonumber \\
& & +\,2\,\Re\left(v_\nu\,,\,T_{\alpha^{-1/2}}R_5u_\nu\right)_{L^2}\,+\,
2\,\Re\left(v_\nu\,,\,R_4w_\nu\right)_{L^2} \nonumber \\
& & +\,2\,\Re\left(v_\nu\,,\,R_3w_\nu\right)_{L^2}\,+\,2\,\Re\left(R_1u_\nu\,,\,w_\nu\right)_{L^2}\,+\,
2\,\Re\left(R_2\d_tu_\nu\,,\,w_\nu\right)_{L^2}\,. \nonumber
\end{eqnarray}

The important issue is that remainders can be controlled in terms of the approximate energy. As a matter
of facts, taking advantage of Proposition \ref{p:symb-calc} and Lemma \ref{l:lzll-symb}, we get the following estimates.
\begin{itemize}
 \item $R_1$ has principal symbol equal to $\d_\xi\left(\alpha(\gamma^2+|\xi|^2)^{1/2}\right)\d_x\d_t(\alpha^{-1/2})$, so
\begin{eqnarray}
 \left|2\,\Re\left(R_1u_\nu\,,\,w_\nu\right)_{L^2}\right| & \leq 
 & C\,(\nu+1)\,e_{\nu}\,. \label{est:R_1}
\end{eqnarray}
 \item The principal symbol of $R_2$ is instead $\d_\xi\left(\alpha(\gamma^2+|\xi|^2)^{1/2}\right)\d_x(\alpha^{-1/2})$, so,
remembering also the control on $\|\d_tu_\nu\|_{L^2}$, we have:
\begin{equation} \label{est:R_2}
 \left|2\,\Re\left(R_2\d_tu_\nu\,,\,w_\nu\right)_{L^2}\right|\,\leq\,C\,\nu\,\left(e_{\nu}\right)^{1/2}\,\|w_\nu\|_{L^2}\,\leq\,
C\,(\nu+1)\,e_{\nu}\,.
\end{equation}
 \item Symbolic calculus tells us that the principal part of $R_3$ is given by
$\d_\xi\d_x\left(\alpha(\gamma^2+|\xi|^2)^{1/2}\right)$, therefore
\begin{equation} \label{est:R_3}
 \left|2\,\Re\left(v_\nu\,,\,R_3w_\nu\right)_{L^2}\right|\,\leq\,C\,\|v_\nu\|_{L^2}\,\nu\,\|w_\nu\|_{L^2}\,\leq\,
C\,(\nu+1)\,e_{\nu}\,.
\end{equation}
 \item Now, $R_4$ has $\d_\xi\left(\alpha^{-1/2}\right)\d_x\left(\alpha^{3/2}(\gamma^2+|\xi|^2)^{1/2}\right)$ as principal symbol, so
\begin{equation} \label{est:R_4}
 \left|2\,\Re\left(v_\nu\,,\,R_4w_\nu\right)_{L^2}\right|\,\leq\,C\,\|v_\nu\|_{L^2}\,\nu\,\|w_\nu\|_{L^2}\,\leq\,
C\,(\nu+1)\,e_{\nu}\,.
\end{equation}
 \item $R_5$ is given, at the highest order, by the product of the symbols
$\d_\xi\left(\alpha^{3/2}(\gamma^2+|\xi|^2)^{1/2}\right)$ and $\d_x\left(\alpha^{1/2}(\gamma^2+|\xi|^2)^{1/2}\right)$, and then we get
\begin{equation} \label{est:R_5}
 \left|2\,\Re\left(v_\nu\,,\,T_{\alpha^{-1/2}}R_5u_\nu\right)_{L^2}\right|\,\leq\,
C\,\|v_\nu\|_{L^2}\,2^\nu\,\nu\,\|u_\nu\|_{L^2}\,\leq\,
C\,(\nu+1)\,e_{\nu}\,.
\end{equation}
\end{itemize}

\subsubsection{Principal part of the operator $L$}

Now, thanks to previous computations, it's natural to pair up the second term of \eqref{eq:d_t-v} with the first one of the last
equality of \eqref{eq:d_t-w}. As $\alpha$ is a symbol of order $0$, we have
$$
\left|2\,\Re\!\left(v_\nu\,,\,T_{\alpha^{-1/2}}\sum_{i,j}\d_i\left(T_{a_{ij}}\d_ju_\nu\right)\right)_{\!\!\!L^2}+
2\,\Re\!\left(v_\nu\,,\,T_{\alpha^{-1/2}}T_{\alpha^2(\gamma^2+|\xi|^2)}u_\nu\right)_{L^2}\right|\,\leq\,
C\,\|v_\nu\|_{L^2}\,\|\zeta_\nu\|_{L^2}\,,
$$
where we have set
\begin{equation} \label{eq:zeta_nu}
\zeta_\nu\,:=\,T_{\alpha^2(\gamma^2+|\xi|^2)}u_\nu\,+\,\sum_{i,j}\d_i\left(T_{a_{ij}}\,\d_ju_\nu\right)\,=\,
\sum_{ij}\,T_{a_{ij, \veps}\xi_i\xi_j+\gamma^2}u_\nu\,+\,\d_i\left(T_{a_{ij}}\,\d_ju_\nu\right)\,.
\end{equation}
We remark that
$$
\d_i\left(T_{a_{ij}}\,\d_ju_\nu\right) = T_{\d_ia_{ij}}\d_ju_\nu\,-\,T_{a_{ij}\xi_i\xi_j}u_\nu,
$$
where, with a little abuse of notation, we have written $\d_ia_{ij}$ meaning that we are taking
the derivative of the classical symbol associated to $a_{ij}$.

First of all, using also spectral localization properties, we have
\begin{eqnarray}
 \left\|T_{\d_ia_{ij}}\d_ju_\nu\right\|_{L^2} & \leq & 
\|S_{\mu}\,\d_ia_{ij}\|_{L^\infty}\,\|S_{\mu+2}\d_ju_\nu\|_{L^2}\,+\,
\sum_{k\geq{\mu+3}} \left\|\nabla S_{k-3}a_{ij}\right\|_{L^\infty}\,
\left\|\Delta_k\nabla u_\nu\right\|_{L^2} \label{est:T_1} \\
& \leq & C\left(\mu+1\right)\,\left(\sup_{i,j}\|a_{ij}\|_{LL_x}\right)\,\|\nabla u_\nu\|_{L^2} \nonumber \\
& & \qquad\qquad\qquad\quad+\,
\sum_{k\geq\mu+3\,,\,k\sim\nu}\left(k+1\right)\,\left(\sup_{i,j}\|a_{ij}\|_{LL_x}\right)\,\|\nabla\Delta_k u_\nu\|_{L^2} \nonumber \\
& \leq & C_\mu\,(\nu+1)\,\left(\sup_{i,j}\|a_{ij}\|_{LL_x}\right)\,\left(e_{\nu}\right)^{1/2}\,, \nonumber
\end{eqnarray}
where $\mu$ is the parameter fixed in \eqref{pd_eq:pp_symb} and we have used also \eqref{est:S_ka}.

Next, we have to control the term
$$
T_{a_{ij,\veps}\xi_i\xi_j+\gamma^2}u_\nu\,-\,T_{a_{ij}\xi_i\xi_j}u_\nu\,=\,T_{\left(a_{ij,\veps}-a_{ij}\right)\xi_i\xi_j}u_\nu\,+\,
T_{\gamma^2}u_\nu\,.
$$
It's easy to see that
$$
\left\|T_{\left(a_{ij,\veps}-a_{ij}\right)\xi_i\xi_j}u_\nu\right\|_{L^2}\,\leq\,
C\,\veps\,\log\left(1+\frac{1}{\veps}\right)2^\nu\,\|\nabla u_\nu\|_{L^2}\,,
$$
and so, keeping in mind that $\veps=2^{-\nu}$,
\begin{equation} \label{est:delta-T}
 \left\|T_{\left(a_{ij,\veps}-a_{ij}\right)\xi_i\xi_j+\gamma^2}u_\nu\right\|_{L^2}\,\leq\,
C_\g\,(\nu+1)\left(e_{\nu}\right)^{1/2}\,.
\end{equation}

Therefore, from \eqref{est:T_1} and  \eqref{est:delta-T} we finally get
\begin{equation} \label{est:sec-ord}
 \left|2\,\Re\left(v_\nu,T_{\alpha^{-1/2}}\sum_{i,j}\d_i\left(T_{a_{ij}}\d_ju_\nu\right)\right)_{L^2}+
2\,\Re\left(v_\nu,T_{\alpha^{-1/2}}T_{\alpha^2(\g^2+|\xi|^2)}u_\nu\right)_{L^2}\right|\,\leq\,C\,(\nu+1)\,e_{\nu}\,,
\end{equation}
where the constant $C$ depends on the log-Lipschitz norm (with respect to space) of the coefficients $a_{ij}$
and on the fixed parameters $\mu$ and $\g$.

\medskip
To sum up, from inequalities \eqref{est:d_t-z}, \eqref{est:Lu_nu}, \eqref{est:zeta} and \eqref{est:sec-ord} and from estimates of
remainder terms \eqref{est:R_1}-\eqref{est:R_5}, we can conclude that
\begin{eqnarray}
 \frac{d}{dt}e_{\nu}(t) & \leq & C_1\,(\nu+1)\,e_{\nu}(t)\,+\,
C_2\,\left(e_{\nu}(t)\right)^{1/2}\,\left\|\left(\wtilde Lu\right)_{\!\!\nu}(t)\right\|_{L^2} \label{est:d_t-e} \\
 & & \qquad\qquad\qquad\qquad\qquad
+\left|2\sum_{i,j}\Re\left(v_\nu\,,\,T_{\alpha^{-1/2}}\d_i\left[\Delta_\nu,T_{a_{ij}}\right]\d_ju\right)_{L^2}\right|\,. \nonumber
\end{eqnarray}

\subsection{Commutator term} \label{ss:comm}
We want to estimate the quantity 
$$
\left|\sum_{i,j}\Re\left(v_\nu\,,\,T_{\alpha^{-1/2}}\d_i\left[\Delta_\nu,T_{a_{ij}}\right]\d_ju\right)_{L^2}\right|\,.
$$

We start remarking that 
$$
[\Delta_\nu, T_{a_{ij}}]w= [\Delta_\nu, S_\mu a_{ij}]S_{\mu+2}w+\sum_{k=\mu+3}^{+\infty}[\Delta_\nu, S_{k-3}a_{ij}]\,\Delta_k w,
$$
where $\mu$ is fixed, as usual (see Remark \ref{r:p-prod}). In fact $\Delta_\nu$ and $\Delta_k$ commute, so that
$$
\Delta_\nu(S_\mu{a_{ij}}S_{\mu+2}w)-
S_\mu{a_{ij}}(S_{\mu+2}\Delta_\nu w)= \Delta_\nu(S_\mu{a_{ij}}S_{\mu+2}w)-
S_\mu{a_{ij}}\Delta_\nu (S_{\mu+2}w),
$$
and similarly
$$
\Delta_\nu(S_{k-3}{a_{ij}}\Delta_kw)-
S_{k-3}{a_{ij}}\Delta_k(\Delta_\nu w)= \Delta_\nu(S_{k-3}{a_{ij}}\Delta_kw)-
S_{k-3}{a_{ij}}\Delta_\nu (\Delta_kw).
$$
Consequently, taking into account also that $S_{k+2}$ and $\Delta_k$ commute with $\partial_j$, we have
$$
\partial_i\left([\Delta_\nu, T_{a_{ij}}]\,\partial_ju\right)\,=\,\partial_i\left(
[\Delta_\nu, S_\mu a_{ij}]\,\partial_j (S_{\mu+2}u)\right)\,+\,
 \partial_i\left(\sum_{k=\mu+3}^{+\infty}[\Delta_\nu, S_{k-3}{a_{ij}}]\, \partial_j (\Delta_ku)\right).
 $$

Let's consider first the term
$$
\partial_i\left(
[\Delta_\nu, S_\mu a_{ij}]\,\partial_j (S_{\mu+2}u)\right).
$$
Looking at the support of the Fourier transform of $[\Delta_\nu, S_\mu a_{ij}]\,\partial_j (S_{\mu+2}u)$,
we have that it is contained in $\{|\xi|\leq 2^{\mu+4}\}$ and  moreover 
$[\Delta_\nu, S_\mu a_{ij}]\,\partial_j (S_{\mu+2}u)$ is identically $0$ if $\nu\geq \mu+5$.
Then, from Bernstein's inequalities and \cite[Th. 35]{Coi-Mey} we have that
$$
\left\|\partial_i\left([\Delta_\nu, S_\mu a_{ij}]\,\partial_j (S_{\mu+2}u)\right)\right\|_{L^2}\,\leq\,
C_\mu\,\left(\sup_{i,j}\|a_{ij}\|_{LL_x}\right)\,\|S_{\mu+2}u\|_{L^2}\,,
$$
hence, putting all these facts together, we have
\begin{eqnarray}
& & \left|\sum_{\nu=0}^{+\infty}\,e^{-2\beta(\nu+1)t}\,2^{-2\nu\theta}\,\sum_{ij}
2\,\Re\!\biggl(v_\nu\,,\,T_{\alpha^{-1/2}}\,\partial_i\left(
[\Delta_\nu, S_\mu a_{ij}]\partial_j (S_{\mu+2}u)\,\right)\biggr)_{L^2}\right| \label{est:comm0} \\
 && \qquad\qquad\qquad \leq\, C_\mu\left(\sup_{i,j}\|a_{ij}\|_{LL_x}\right)\,
\sum_{\nu=0}^{\mu+4}e^{-2\beta(\nu+1)t}\,2^{-2\nu\theta}\|v_\nu\|_{L^2}\left(\sum_{h=0}^{\mu+2}\|u_h\|_{L^2}\right) \nonumber\\
&& \qquad\qquad\qquad\leq\, C_\mu\,\left(\sup_{i,j}\|a_{ij}\|_{LL_x}\right)\, e^{\beta(\mu +5)T}\,2^{(\mu+4)\theta} \nonumber \\
&& \qquad\qquad\qquad\qquad
\times\left(\sum_{\nu=0}^{\mu+4}\,e^{-\beta(\nu+1)t}\,2^{-\nu\theta}\|v_\nu\|_{L^2}\right)
\left(\sum_{h=0}^{\mu+4}\,e^{-\beta(h+1)t}\,2^{-h\theta}\|u_h\|_{L^2}\right) \nonumber \\
 && \qquad\qquad\qquad \leq\, C_\mu\,\left(\sup_{i,j}\|a_{ij}\|_{LL_x}\right)\,e^{\beta(\mu +5)T}\,2^{(\mu+4)\theta}
\sum_{\nu=0}^{\mu+4} e^{-2\beta(\nu+1)t}\,2^{-2\nu\theta}e_\nu(t)\,. \nonumber
\end{eqnarray}

Next, let's consider 
$$
\partial_i\left(\sum_{k=\mu+3}^{+\infty}[\Delta_\nu, S_{k-3}{a_{ij}}]\,\partial_j(\Delta_k u)\right).
$$
Looking at  the support of the Fourier transform, it is possible to see that
$$
[\Delta_\nu, S_{k-3}a_{ij}]\,\partial_j(\Delta_k u)
$$
is identically 0 if $|k-\nu|\geq 3$. Consequently the sum over $k$ is reduced to at most 5 terms:
$\partial_i([\Delta_\nu, S_{\nu-5}a_{ij}]\,\partial_j(\Delta_{\nu-2} u))+
\dots +\partial_i([\Delta_\nu, S_{\nu-1}a_{ij}]\,\partial_j(\Delta_{\nu+2} u))$,
each of them having the support of the Fourier transform contained in $\{|\xi|\leq 2^{\nu+1}\}$. 
Let's consider one of these terms, e.g. $\partial_i([\Delta_\nu, S_{\nu-3}a_{ij}]\,\partial_j(\Delta_{\nu} u))$,
the computation for the other ones being similar. We have, from Bernstein's inequalities,
$$
\left\|\partial_i\left([\Delta_\nu, S_{\nu-3}a_{ij}]\,\partial_j(\Delta_{\nu} u)\right)\right\|_{L^2}\,
\leq\, C\,2^\nu\, \|[\Delta_\nu, S_{\nu-3}a_{ij}]\,\partial_j(\Delta_{\nu} u)\|_{L^2}\,.
$$
On the other hand, using \cite[Th. 35]{Coi-Mey} again, we have:
$$
\|[\Delta_\nu, S_{\nu-3}a_{ij}]\partial_j(\Delta_{\nu} u)\|_{L^2}\,\leq\, C\,
\|\nabla S_{\nu-3}a_{ij}\|_{L^\infty}\,\|\Delta_\nu u\|_{L^2}\,,
$$
where $C$ does not depend on $\nu$. Consequently, using also (\ref{est:S_ka}), we deduce
$$
\left\|\partial_i\left([\Delta_\nu, S_{\nu-3}a_{ij}]\,\partial_j(\Delta_{\nu} u)\right)\right\|_{L^2}
\,\leq\, C\, 2^\nu\, (\nu+1)\,\left(\sup_{i,j}\|a_{ij}\|_{LL_x}\right)\,\|\Delta_{\nu} u\|_{L^2}\,.
$$
From this last inequality and similar ones for the other terms, we infer
$$
\left|\sum_{i,j}\Re\left(v_\nu\,,\,T_{\alpha^{-1/2}}\partial_i
\left(\sum_{k=\mu+3}^{+\infty}[\Delta_\nu, S_{k-3}{a_{ij}}]\,\partial_j(\Delta_k u)\right)\right)_{L^2}\right|
\,\leq\, C\, \left(\sup_{i,j}\|a_{ij}\|_{LL_x}\right)\, (\nu+1) \,e_{\nu}(t)
$$
and then
\begin{eqnarray}
& & \left|\sum_{\nu=0}^{+\infty}\,e^{-2\beta(\nu+1)t}\,2^{-2\nu\theta}\,\sum_{ij}
2\,\Re\!\left(v_\nu\,,\,T_{\alpha^{-1/2}}\partial_i\left(\sum_{k=\mu+3}^{+\infty}[\Delta_\nu, S_{k-3}{a_{ij}}]\,
\partial_j(\Delta_k u)\right)\right)_{L^2}\right| \label{est:comm1} \\
 & & \qquad\qquad\qquad\qquad\qquad\qquad\qquad\quad\,\leq\,
C \,\left(\sup_{i,j}\|a_{ij}\|_{LL_x}\right)\,\sum_{\nu=0}^{+\infty}\,(\nu+1)\,e^{-2\beta(\nu+1)t}\,2^{-2\nu\theta}\,e_{\nu}(t)\,. \nonumber
\end{eqnarray}
Collecting the informations from \eqref{est:comm0} and \eqref{est:comm1}, we finally obtain
\begin{eqnarray}
& & \left|\sum_{\nu=0}^{+\infty}\,e^{-2\beta(\nu+1)t}\,2^{-2\nu\theta}\,\sum_{ij}
2\,\Re\!\left(v_\nu\,,\,T_{\alpha^{-1/2}}\d_i\left[\Delta_\nu,T_{a_{ij}}\right]\d_ju\right)_{L^2}\right| \label{est:comm} \\
 & & \qquad\qquad\qquad\qquad\qquad\qquad\qquad\qquad\,\leq\,
C_3\,\sum_{\nu=0}^{+\infty}\,(\nu+1)\,e^{-2\beta(\nu+1)t}\,2^{-2\nu\theta}\,e_{\nu}(t)\,, \nonumber
\end{eqnarray}
where $C_3$ depends on $\mu$, $\sup_{i,j}\|a_{ij}\|_{LL_x}$, on $\theta$ and on the product $\beta\,T$.
\subsection{Final estimate} \label{ss:f-est}

From \eqref{est:d_t-e} and \eqref{est:comm} we get
\begin{eqnarray*}
 \frac{d}{dt}E(t) & \leq & \left(C_1\,+\,C_3\,-\,2\beta\right)
\sum_{\nu=0}^{+\infty}\,(\nu+1)\,e^{-2\beta(\nu+1)t}\,2^{-2\nu\theta}\,e_{\nu}(t) \\
& & \qquad\qquad+\,C_2\,\sum_{\nu=0}^{+\infty}e^{-2\beta(\nu+1)t}\,2^{-2\nu\theta}\left(e_{\nu}(t)\right)^{1/2}\,
\left\|\left(\wtilde Lu(t)\right)_{\!\!\nu}\right\|_{L^2} \\
 & \leq & \left(C_1\,+\,C_3\,-\,2\beta\right)
\sum_{\nu=0}^{+\infty}\,(\nu+1)\,e^{-2\beta(\nu+1)t}\,2^{-2\nu\theta}\,e_{\nu}(t) \\
& & +\,C_2\,\sum_{\nu=0}^{+\infty}e^{-2\beta(\nu+1)t}\,2^{-2\nu\theta}\left(e_{\nu}(t)\right)^{1/2}\,
\left\|\left(\sum_{i,j}\d_i\left((a_{ij}-T_{a_{ij}})\d_ju\right)\right)_{\!\!\!\!\nu} \right\|_{L^2} \\
& & \qquad+\,C_2\,\sum_{\nu=0}^{+\infty}e^{-2\beta(\nu+1)t}\,2^{-2\nu\theta}\left(e_{\nu}(t)\right)^{1/2}\,
\left\|( Lu(t))_\nu \right\|_{L^2}\,.
\end{eqnarray*}
Now, applying H\"older inequality for series implies
\begin{eqnarray*}
&&\sum_{\nu=0}^{+\infty}e^{-2\beta(\nu+1)t}\,2^{-2\nu\theta}\left(e_{\nu}(t)\right)^{1/2}\,
\left\|\left(\sum_{i,j}\d_i\left((a_{ij}-T_{a_{ij}})\d_ju\right)\right)_{\!\!\!\!\nu} \right\|_{L^2} \\
&& \qquad\qquad\qquad\leq\,
\left(\sum_{\nu=0}^{+\infty}(\nu+1)\,e^{-2\beta(\nu+1)t}\,2^{-2\nu\theta}\,e_{\nu}(t)\right)^{1/2} \\
&& \qquad\qquad\qquad\qquad\times \left(\sum_{\nu=0}^{+\infty}e^{-2\beta(\nu+1)t}\,2^{-2\nu\theta}\,(\nu+1)^{-1}
\left\|\left(\sum_{i,j}\d_i\left((a_{ij}-T_{a_{ij}})\d_ju\right)\right)_{\!\!\!\!\nu}\right\|^2_{L^2}\right)^{1/2}\,,
\end{eqnarray*}
and, by definition, one has
\begin{eqnarray*}
&&\hspace{-1cm}
\left(\sum_{\nu=0}^{+\infty}e^{-2\beta(\nu+1)t}\,2^{-2\nu\theta}\,(\nu+1)^{-1}
\left\|\left(\sum_{i,j}\d_i\left((a_{ij}-T_{a_{ij}})\d_ju\right)\right)_{\!\!\!\!\nu}\right\|^2_{L^2}\right)^{1/2} \\
&& \qquad\qquad\qquad\qquad\qquad\qquad\qquad\qquad\quad
=\,\left\|\sum_{i,j}\d_i\left((a_{ij}-T_{a_{ij}})\d_ju\right)\right\|_{H^{-\theta-\beta^*t-{\frac{1}{2}}\log}}\,.
\end{eqnarray*}
From \cite[Prop. 3.4]{C-M} we have that
\begin{equation}
\left\| \sum_{i,j}\d_i\left((a_{ij}-T_{a_{ij}})\d_ju\right)\right\|_{H^{-s-{\frac{1}{2}}\log}}\,\leq\, C\,
\left(\sup_{i,j}\|a_{ij}\|_{LL_x}\right)\,\|u\|_{H^{1-s+{\frac{1}{2}}\log}}\,,  \label{a-T_a_est}
\end{equation}
with $C$ uniformly bounded for $s$ in a compact set of $]0,1[$. Consequently,
\begin{eqnarray*}
&&\left(\sum_{\nu=0}^{+\infty}e^{-2\beta(\nu+1)t}\,2^{-2\nu\theta}\,(\nu+1)^{-1}
\left\|\left(\sum_{i,j}\d_i\left((a_{ij}-T_{a_{ij}})\d_ju\right)\right)_{\!\!\!\!\nu}\right\|^2_{L^2}\right)^{1/2} \\
&& \qquad\qquad\qquad\qquad\qquad\qquad\qquad\qquad\qquad
\leq\, C\,\left(\sup_{i,j}\|a_{ij}\|_{LL_x}\right)\,\|u\|_{H^{1-\theta-\beta^*t+{\frac{1}{2}}\log}} \\
&& \qquad\qquad\qquad\qquad\qquad\qquad\qquad\qquad\qquad
\leq\,C\left(\sum_{\nu=0}^{+\infty}(\nu+1)\,e^{-2\beta(\nu+1)t}\,2^{-2\nu\theta}\,
e_{\nu}(t)\right)^{1/2}\,,
\end{eqnarray*}
and finally
\begin{eqnarray*}
&& \sum_{\nu=0}^{+\infty}e^{-2\beta(\nu+1)t}\,2^{-2\nu\theta}\left(e_{\nu}(t)\right)^{1/2}\,
\left\|\left(\sum_{i,j}\d_i\left((a_{ij}-T_{a_{ij}})\d_ju\right)\right)_{\!\!\!\!\nu} \right\|_{L^2} \\ 
&&\qquad\qquad\qquad\qquad\qquad\qquad\qquad\qquad\qquad
\leq\, C_4\, \sum_{\nu=0}^{+\infty}(\nu+1)e^{-2\beta(\nu+1)t}\,2^{-2\nu\theta}\,e_{\nu}(t)\,,
\end{eqnarray*}
with $C_4$ uniformly bounded for $\beta^* t +\theta$ in a compact set of $]0,1[\,$. So, if we take $\beta>0$
(recall that $\beta^*=\beta(\log2)^{-1}$) and $T\in\,]0,T_0]$ such that
\begin{equation} \label{eq:T}
\beta^*\,T\;=\;\delta\;<\;1-\theta\,,
\end{equation}
we have $0<\theta\leq\theta+\beta^*t\leq\theta+\delta<1$. Therefore we obtain
\begin{eqnarray*}
 \frac{d}{dt}E(t) & \leq & \left(C_1\,+C_4\,C_2\,+\,C_3-\,2\beta\right)
\sum_{\nu=0}^{+\infty}\,(\nu+1)\,e^{-2\beta(\nu+1)t}\,2^{-2\nu\theta}\,e_{\nu}(t) \\
& & \qquad\qquad\qquad\qquad
+\,C_2\,\sum_{\nu=0}^{+\infty}e^{-2\beta(\nu+1)t}\,2^{-2\nu\theta}\left(e_{\nu}(t)\right)^{1/2}\,\left\|( Lu(t))_\nu\right\|_{L^2}\,.
\end{eqnarray*}
Now we fix $\beta$ large enough, such that $C_1\,+\,C_4\,C_2\,+\,C_3-\,2\beta\,\leq\,0$: this corresponds to take $T>0$
small enough in \eqref{eq:T}. Therefore we finally arrive to the estimate
$$
\frac{d}{dt}E(t)\,\leq\,C_2\,\left(E(t)\right)^{1/2}\,\left\| Lu(t)\right\|_{H^{-\theta-\beta^*t}}\,;
$$
applying Gronwall's Lemma and keeping in mind \eqref{est:E(0)} and \eqref{est:E(t)} give us estimate \eqref{est:thesis}. \qed

\begin{rem} \label{r:T}
 Let us point out that relation \eqref{eq:T} gives us a condition on the lifespan $T$ of a solution to the Cauchy problem
for \eqref{eq:op}. It depends on $\theta\in\,]0,1[$ and on $\beta^*>0$, hence on constants $C_1\ldots C_4$.
Going after the guideline of the proof, one can see that, in the end, the time $T$ depends only on the
index $\theta$, on the parameter $\mu$ defined by conditions \eqref{est:param}, on constants $\lambda_0$ and $\Lambda_0$
defined by \eqref{h:hyp} and on the quantities $\sup_{i,j}\left|a_{ij}\right|_{LZ_t}$ and $\sup_{i,j}\left|a_{ij}\right|_{LL_x}$.
\end{rem}

\end{document}